\documentclass[10pt,a4paper,reqno]{amsart} 

\usepackage{amsmath}
\usepackage{mathtools}
\usepackage{amssymb}
\usepackage{graphicx}
\usepackage{color}
\usepackage{latexsym}
\usepackage[T1]{fontenc}

\newcommand{\zs}{{\mathbb Z}} 
\newcommand{\cs}{{\mathbb C}} 

\newcommand{\al}{\alpha}

\newcommand{\si}{\sigma}

\newcommand{\eps}{\epsilon}
\newcommand{\sig}{\sigma}

\newcommand{\bu}{\bar u}
\newcommand{\bv}{\bar v}
\newcommand{\bw}{\bar w}
\newcommand{\bU}{\bar U}
\newcommand{\bV}{\bar V}
\newcommand{\bq}{\bar q}

\newcommand{\PP}{\mathbb{P}}

\newcommand{\II}{{\sf I}}

\DeclareMathOperator{\id}{id}
\DeclareMathOperator{\Res}{Res}
\DeclareMathOperator{\SUM}{SUM}

\newtheorem{Theorem}{Theorem}
\newtheorem{Proposition}[Theorem]{Proposition}

\newtheorem{Corollary}[Theorem]{Corollary}

\newtheorem{Lemma}[Theorem]{Lemma}

\newcommand{\beq}{\begin{equation}}
\newcommand{\eeq}{\end{equation}}

\newcommand{\gf}{generating function}

\def\emm#1,{{\em #1}}

\newcommand{\Sn}{{\mathfrak S}}
\newcommand{\cA}{{\mathfrak A}}
\newcommand{\p}{permutation}
\newcommand{\ps}{permutations}

\catcode`\@=11
\def\section{\@startsection{section}{1}%
 \z@{.7\linespacing\@plus\linespacing}{.5\linespacing}%
 {\normalfont\bfseries\scshape\centering}}

\def\subsection{\@startsection{subsection}{2}%
  \z@{.5\linespacing\@plus\linespacing}{.5\linespacing}%
  {\normalfont\bfseries\scshape}}

\def\subsubsection{\@startsection{subsubsection}{3}%
 \z@{.5\linespacing\@plus\linespacing}{-.5em}
  {\normalfont\bfseries\itshape}}
\catcode`\@=12

\def\qed{$\hfill{\vrule height 3pt width 5pt depth 2pt}$}

\title[The expected  number of inversions after $n$ adjacent
 transpositions]{The expected number of inversions\\
after $n$ adjacent transpositions}

\author{Mireille Bousquet-M\'elou}
\address{CNRS, LaBRI, Universit\'e Bordeaux 1,
351 cours de la Lib\'eration,
  33405 Talence Cedex, France}
\email{mireille.bousquet@labri.fr}

\date{\today}

\begin{document}

\begin{abstract}
We give a new expression for the expected number of inversions in the
product of $n$ random adjacent transpositions in the symmetric group
$\Sn_{m+1}$. We then  derive from this expression the asymptotic behaviour of this number
when $n\equiv n_m$ scales with $m$ in various ways.
 Our starting point is an equivalence, due to Eriksson
\emm et al.,, with a problem of
weighted walks
confined to a triangular area of the plane.

\end{abstract}

\maketitle

\section{Introduction}
Let $\Sn_{m+1}$ denote the group of permutations on the set $\{0,1, 2,
\ldots , m\}$. 
Define the adjacent  transposition $s_i$ to be the two-cycle $(i,i+1)$.
We consider the \emm adjacent transposition Markov chain,  $\pi^{(0)}, \pi^{(1)}, \pi^{(2)}, \ldots$  on $\Sn_{m+1}$. That is,  $\pi^{(0)}$ is the identity, and $\pi^{(n+1)}= \pi^{(n)}s_i$, where $i \in\{0, \ldots, m-1 \}$ is chosen uniformly at random. 
For instance, denoting $\pi=\pi_0\pi_1\cdots \pi_m$ if  $\pi(i)=\pi_i$ for all $i$, we may have, when $m=3$,
$$
\begin{array}{lllll}
  \pi^{(0)}&=& 0123 &=& \id,
\\
  \pi^{(1)}&=& 0213 &=&\pi^{(0)} s_1,
\\
  \pi^{(2)}&=& 2013 &=& \pi^{(1)} s_0,
\\
  \pi^{(3)}&=& 2103 &=&\pi^{(2)} s_1.
\end{array}
$$

Markov chains on finite groups 
have attracted a lot of interest in, at least, the past 30 years, at the interface 
of algebra and probability theory~\cite{aldous,diaconis-book,diaconis-saloff-comparison,diaconis-shahshahani-transpositions,saloff-coste-notices,saloff-coste-survey}. Such chains are also studied in
computational biology, in connection with models of gene mutations~\cite{b-d-general,durrett-book,eriksen-hultman,eriksson,wang-warnow}. A central question is to
estimate the \emm mixing time, of the chain, for large values of
$m$. (The mixing time is the number of steps the chain takes to
approach its equilibrium distribution.) 
 The above chain is periodic of period 2, since $\pi^{(2n)}$ is always
 in the alternating group.  Thus it has no equilibrium
 distribution. However, an elementary modification
 (which consists in choosing $\tilde \pi^{(n+1)}= \tilde \pi^{(n)}$ with probability $1/(m+1)$ and
 otherwise multiplying by a transposition $s_i$ chosen uniformly)
 makes it aperiodic, and the equilibrium distribution is then uniform
 over $\Sn_{m+1}$. 
The mixing time is known to be $\Theta(m^3\log m)$ (see Aldous~\cite{aldous},
 Diaconis and Saloff-Coste~\cite{diaconis-saloff-comparison}
and Wilson~\cite{wilson}).

More recently, Berestycki and Durrett~\cite{b-d-adjacent}  studied a continuous time
version of this chain, denoted $(X_t^m)_{t\ge 0}$, where  multiplications by
adjacent transpositions occur according to a Poisson process of rate
1. The connection with the chain $\pi^{(n)}$
described above is straightforward. The authors focussed on the inversion number
$D_t^m$ of $X_t^m$, as a function of $m$ and $t$. They
established the convergence in probability of this random variable,
suitably normalized, when $t\equiv t_m$ scales with $m$ in various ways. The
limit is usually described in probabilistic terms, except in one simple
regime, $m  \ll t \ll m^3$, where $D_t^m/\sqrt{mt}$ is shown to converge to
$\sqrt{2/\pi} $. Of course, the interesting regimes are those
occurring \emm before, stationarity.

This paper actually stems from a remark made in Berestycki \& Durrett's paper. They
quote a paper by Eriksen~\cite{niklas}, which gives a rather formidable
expression for $\II_{m,n}$, the expected number of
inversions in $\pi^{(n)}$: 
\beq\label{niklas-expr}
\II_{m,n}
= \sum_{r=1}^n \frac 1 {m^r} {n\choose r}
\sum_{s=1}^r{{r-1}\choose {s-1}}(-4)^{r-s}g_{s,m} \,h_{s,m},
\eeq
with
$$
g_{s,m}= \sum_{\ell =0} ^m \sum_{k\ge 0} (-1)^k (m-2\ell) 
{{2\lceil s/2\rceil -1}\choose{\lceil s/2\rceil+\ell +k(m+1)}}
$$
and 
$$
h_{s,m}=\sum_{j\in \zs} (-1)^j 
{{2\lfloor s/2\rfloor }\choose{\lfloor s/2\rfloor+j(m+1)}}.
$$
The authors  underline that ``it is far from obvious how to extract
useful asymptotic from this formula''. This observation motivated our
interest in the expectation 
$\II_{m,n}$. As the problem is of an algebraic nature, it should  have
much structure: would it be possible to find an alternative expression of
$\II_{m,n}$, with a neater structure than~\eqref{niklas-expr}, and to
derive  asymptotic results from this expression?

This is precisely what we do in this paper. Our
alternative formula  for $\II_{m,n}$ reads as follows. 

\begin{Theorem}\label{thm:main}
Let $m\ge 1$. For $k \in \zs$, denote 
\beq\label{notation}
\al_k= \frac {(2k+1)\pi}{2m+2}, \quad c_k=\cos \alpha_k, \quad 
 s_k= \sin \alpha_k \quad  \hbox{  and } \quad
x_{jk}= 1-\frac 4 m (1-c_jc_k).
\eeq
The expected number of inversions after $n$ adjacent
transpositions in $\Sn_{m+1}$ is
$$
\II_{m,n}= \frac{m(m+1)} 4 -\frac 1 {8(m+1)^2}\sum_{k,j=0}^m
\frac{(c_j+c_k)^2}{s_j^2s_k^2} 
\ {x_{jk}}^n.
$$
Equivalently, the \gf\ $\II_m(t)=\sum_{n\ge 0} \II_{m,n} t^n$ is
\beq\label{GF1}
\II_m(t)= 
 \frac{m(m+1)} {4(1-t)}-\frac 1 {8(m+1)^2}\sum_{k,j=0}^m
\frac{(c_j+c_k)^2}{s_j^2s_k^2} \frac 1{1-tx_{jk}}.
\eeq
\end{Theorem}
This result is proved in Section~\ref{sec:exact}, and asymptotic
results are derived in Sections~\ref{sec:small} to \ref{sec:int} for
 three main regimes: linear ($n_m=\Theta(m)$) and before, cubic  ($n_m=\Theta(m^3)$) and beyond,
intermediate ($m\ll n_m \ll m^3$).
For the moment, let us give a few comments and variants on
Theorem~\ref{thm:main}.

\medskip

\noindent {\bf Limit behaviour.}
The chain $\pi^{(n)}$ has period 2, as $\pi^{(n)}$ is an even (resp. odd)
\p\ if $n$ is even (resp. odd). But the sub-chains $\pi^{(2n)}$
and $\pi^{(2n+1)}$ are aperiodic on their respective state spaces,
the alternating group $\cA_{m+1}$ (the group of even permutations) for
$\pi^{(2n)}$ and its complement
$\Sn_{m+1}\setminus\cA_{m+1}$ for $\pi^{(2n+1)}$. Moreover, each of
these chains is irreducible and symmetric, and thus admits the uniform
distribution as its equilibrium distribution. 
For $m\ge 3$, the average number of inversions of an
element of $\cA_{m+1}$ 
 is
$m(m+1)/4$, and the same is true of elements of
 $\Sn_{m+1}\setminus\cA_{m+1}$. Thus when $m\ge 3$ is fixed and
 $n\rightarrow \infty$, we expect 
$\II_{m,n}$ to tend to $m(m+1)/4$.  This  can be
seen on the above expression of $\II_{m,n}$, upon observing that, for $m\ge
3$ and $0\le j, k 
\le m$ with $j+k\not = m$, there holds $x_{jk} \in (-1,1)$. (The
condition  $j+k\not = m$ is equivalent to $c_j+c_k\not = 0$.)
If $m\ge 8$, the stronger property $x_{jk} \in (0,1)$ holds, which
shows that  $\II_{m,n}$ is an increasing function of $n$.

\medskip
\noindent {\bf Eigenvalues of the transition matrix.}
Another consequence of the above theorem is that we have identified
a quadratic number  of eigenvalues of the transition matrix of the chain
(see~\cite{niklas} for a   detailed account of the connection 
between these eigenvalues and $\II_{m,n}$).

\begin{Corollary}\label{coro:eigenvalues}
 Let $c_k=\cos \frac {(2k+1)\pi}{2m+2}$.
 The transition matrix of the adjacent transposition Markov chain on
  $\Sn_{m+1}$ admits ---among others--- the following
  eigenvalues:
$$
x_{jk}= 1-\frac 4 m (1-c_jc_k),
$$
for $0\le j, k\le m$ and $j+k\not = m$.
\end{Corollary}
This result is not really new: as pointed to us by David Wilson, one can derive from the proof of
Lemma~9 in~\cite{wilson} 
that for $0\le p ,q\le m$, with $p\not = q$,
$$
1-\frac 2 m \left(2- \cos \frac{p\pi}{m+1}-\cos \frac{q\pi}{m+1}\right)
$$
is also an eigenvalue. This collection is larger than the one we have
exhibited, as can be seen using
$$
2 c_jc_k= \cos \frac{(k-j)\pi}{m+1}+\cos \frac{(k+j+1)\pi}{m+1}.
$$
The transition matrix has still more eigenvalues, for instance 
$-1$, with eigenvector $(\varepsilon (\si))_{ \si
  \in \Sn_{m+1}}$ where $\varepsilon$ denotes the signature.
The paper~\cite{edelman-white}, originally motivated
by a coding problem, gives a linear number of
eigenvalues of the form $i/m$, where $i$ is an integer.
There exists a description of \emm all, eigenvalues  in terms of the characters of $\Sn_{m+1}$
(see~\cite[Thm.~3]{diaconis-shahshahani-transpositions}). This description is valid
in a much more general framework, and it seems that the complete list of
eigenvalues is not explicitly known.

\medskip
\noindent {\bf Rationality of $\II_m(t)$.}  That the series $\II_m(t)$ is
always rational should not be a surprise: this property  is clear when
considering the 
transition matrix of the chain. Here are the first few values:
\begin{eqnarray*}
  \II_1(t) &=& \frac{t}{(1-t)(1+t)} ,
\\
  \II_2(t) &=&  {\frac {t( 2+t) }{ (1- t )  ( 2-t )  ( 1+t ) }}
,
\\
  \II_3(t) &=& {\frac {3t(  27+9\,t-7\,{t}^{2}-{t}^{3} ) }
{ ( 1-t )  ( 9 +6\,t-{t}^{2})  (9 -6\,t-{t}^{2} ) }}
 ,
\\
  \II_4(t) &=&  {\frac {t( 256-192\,t-48\,{t}^{2}+44\,{t}^{3}-5\,{t}^{4} )}
{ (1- t )  (16- 5\,{t}^{2} )  (16-20\,t+ 5\,{t}^{2} ) }}
.
\\
\end{eqnarray*}

\medskip
\noindent {\bf A related aperiodic chain.}
If we consider instead the aperiodic variant of the chain  
(obtained by choosing $\tilde \pi^{(n+1)}= \tilde \pi^{(n)}$ with probability $1/(m+1)$ and
 otherwise multiplying by a transposition $s_i$ chosen uniformly), it
 is easy  to see that the expected number of inversions after $n$ steps
 is now
$$
\tilde \II_{m,n}=\sum_{k=0}^n {n\choose k} \frac {m^k}{(m+1)^n}\  \II_{m,k},
$$
so that the associated \gf\ is also rational:
$$
\tilde \II_m(t):=\sum_{n\ge 0}\tilde \II_{m,n}t^n = \frac 1 {1-t/(m+1)} \ \II_m\left( \frac{tm}{m+1-t}\right)
.$$
More generally, if  a transposition occurs with probability $p$, the
\gf\ of the expected number of inversions is
$$
 \frac 1 {1-t(1-p)} \ \II_m\left( \frac{tp}{1-t(1-p)}\right).
$$
(Above, $p=m/(m+1)$, but  in~\cite{wilson} for instance,  $p=1/2$.)

\medskip
\noindent{\bf Alternative expressions.}
 Theorem~\ref{thm:main}  gives a partial fraction expansion of $\II_m(t)$. One of the advantages of~\eqref{GF1} is
that the coefficients involved in the sum over $j$ and $k$ are
negative, which makes the list of  poles of $\II_m(t)$  clear (we
have actually already used this to state Corollary~\ref{coro:eigenvalues}).  
A number of variations are possible, and we will
 use some of them in the asymptotic study of the numbers
$\II_{m,n}$. For instance:
\begin{eqnarray}
  \II_m(t)&= &
 \frac{m(m+1)} {4(1-t)}-\frac 1 {8(m+1)^2}\sum_{k,j=0}^m
\frac{c_j+c_k}{(1-c_j)(1-c_k)} \frac 1{1-tx_{jk}}
\label{I-ser-2}\\
&=&
\frac 1 {8(m+1)^2}\sum_{k,j=0}^m
\frac{c_j+c_k}{(1-c_j)(1-c_k)} \left(\frac1 {1-t}- \frac 1{1-tx_{jk}}\right).
\label{I-ser-3}
\end{eqnarray}
Both formulas are proved in Section~\ref{sec:exact}.

\section{The expected number of inversions}
\label{sec:exact}

We prove in this section Theorem~\ref{thm:main} and its
variants~(\ref{I-ser-2}-\ref{I-ser-3}). Our starting point is a
functional equation satisfied by a series related to $\II_m(t)$. 
We then solve this equation using the bivariate \emm kernel method,, which
has proved successful in the past few years in various enumerative
problems related to lattice paths, permutations, and other
combinatorial objects~\cite{bousquet-motifs,Bous05,bousquet-mishna,bousquet-xin,Rensburg-Prellberg-Rechni,Mishna-Rechni,xin-zhang}.   
The equation under consideration here can
 be interpreted in terms of weighted lattice paths confined to a
triangular portion of the square lattice.  Another problem   of
walks in a triangle, also related to the adjacent transposition Markov chain, was studied  
by Wilson~\cite{wilson} via a diagonalization of the adjacency matrix. The enumeration of (unweighted) walks in a triangle
was performed by Flajolet~\cite{flajolet-stacks} using the reflection principle, in connection with a problem of storage
raised by Knuth~\cite[Ex.~2.2.2.13]{knuth}.

\subsection{A functional equation}\label{sec:characterization}
For $\pi\in \Sn_{m+1}$, let us write $\pi=\pi_0\pi_1\cdots \pi_m$ if
  $\pi(i)=\pi_i$ for all $i$.
For $n\ge 0$ and $0\le i \le j <m$, let
$$
p_{i,j}^{(n)}= \PP(\pi^{(n)}_i > \pi^{(n)}_{j+1}).
$$
Then the expected  number of inversions in $\pi^{(n)}$ is
\beq\label{inv-pijn}
\II_{m,n}= \sum_{0\le i \le j<m} p_{i,j}^{(n)}.
\eeq

Examining how the numbers $p_{i,j}^{(\cdot)}$ may change at the
$n^{\hbox{\small th}}$
step gives a recurrence 
relation for these numbers, first obtained by Eriksson \emm et
al,.~\cite{eriksson}. As shown by the lemma below, it
converts  our problem into the study of weighted walks confined to a
triangular region of the square lattice.
Consider the subgraph $G_m$ of the square lattice $\zs\times \zs$ induced by
the points $(i,j)$, with $0\le i \le j<m$ (Figure~\ref{fig:graph}). We
use the notation $(i,j) \leftrightarrow (k,\ell)$ to mean that the
 points $(i,j)$ and $(k,\ell)$ are adjacent  in this graph.

\begin{figure}[htb]
\begin{center}
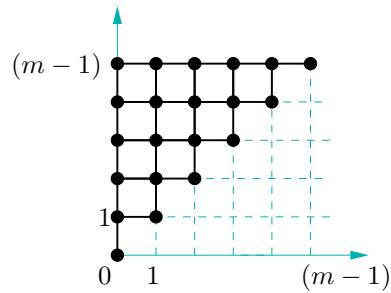
\end{center}
\vskip -5mm\caption{The graph $G_m$.}
\label{fig:graph}
\end{figure}

\begin{Lemma}[\cite{eriksson}]
\label{lem:triangle}  The inversion probabilities $p_{i,j}^{(n)}$ are
characterized by the   following recursion:
$$
p_{i,j}^{(0)}=0 \quad \hbox{for} \quad  0\le i\le j<m,
$$
and for $n\ge 0$,
$$
p_{i,j}^{(n+1)}= p_{i,j}^{(n)}+ \frac 1 m \sum_{(k,\ell)
  \leftrightarrow (i,j)} \left( p_{k,\ell}^{(n)}-p_{i,j}^{(n)}\right) 
+\frac{\delta_{i,j}} m \left( 1-2p_{i,j}^{(n)}\right),
$$
where $\delta_{i,j}= 1$ if $i=j$ and $0$ otherwise.
\end{Lemma}

As often, it is convenient to handle   the numbers
$p_{i,j}^{(n)}$ via their \gf:
$$
P(t;u,v) \equiv P(u,v) := \sum_{n\ge 0} t^n \sum_{0\le i \le j <m}
p_{i,j}^{(n)} u^i v^j.
$$
Multiplying the above recursion by $t^{n+1}$, and then summing over
$n$, gives the following functional equation for $P(u,v)$.

\begin{Corollary}
The series $P(u,v)$ satisfies
\begin{multline}\label{eqP}
\left( 1-t +\frac t m (4-u-\bu-v-\bv)\right) P(u,v)=\\
\frac  t m \left(
\frac {1-u^m v^m}{(1-uv)(1-t)}
-(\bu -1) P_\ell(v) -(v-1)v^{m-1}P_t(u) -(u+\bv) P_d(uv)
\right),
\end{multline}
where $\bu=1/u$, $\bv=1/v$, and the series $P_\ell$, $P_t$ and $P_d$
 describe the numbers 
$p_{i,j}^{(n)}$ on the three borders (left, top, and diagonal) of the graph
$G_m$:
\begin{eqnarray}
  P_\ell(v)
&=&\sum_{n\ge 0}  t^n \sum_{0 \le j <m}
p_{0,j}^{(n)} v^j, \nonumber \\
P_t(u)&=&\sum_{n\ge 0}  t^n \sum_{0 \le i <m}
p_{i,m-1}^{(n)} u^i,\nonumber \\
P_d(u)&= &\sum_{n\ge 0}  t^n \sum_{0 \le i <m}
p_{i,i}^{(n)} u^i. \label{Pd-def}
\end{eqnarray}
\end{Corollary}
In view of~\eqref{inv-pijn}, the \gf\ we are interested in is
$$
\II_m(t)=\sum_{n\ge 0} \II_{m,n} t^n = P(1,1),
$$
which, according to the functional equation~\eqref{eqP}, may be rewritten 
\beq\label{Im-Pd}
\II_m(t)= \frac t{(1-t)^2} -\frac {2t P_d(1)}{m(1-t)}.
\eeq
In the next subsection, we solve~\eqref{eqP}, at least to the point where
we obtain a closed form expression of  $P_d(1)$, and hence a
closed form expression of $\II_m(t)$, as announced in Theorem~\ref{thm:main}.

\subsection{Solution of the functional equation}
We first establish a symmetry property of the series $P(u,v)$.
\begin{Lemma}\label{lem:symm}
  The series $P(u,v)$ satisfies
$$
P(u,v)=u^{m-1} v^{m-1} P(\bv, \bu)
$$
with $\bu=1/u$, $\bv=1/v$. In particular, the ``diagonal'' \gf\ $P_d(u)$
satisfies
\beq\label{symm-diag}
P_d(u)=u^{m-1}P_d(\bu).
\eeq
\end{Lemma}
\begin{proof}
This can be derived from the functional equation satisfied by
$P(u,v)$, but we prefer to give a combinatorial (or probabilistic)
argument.

  Let $\tau$ be the \p \ of $\Sn_{m+1}$ that sends $k$ to $m-k$ for
  all $k$. Note that $\tau$ is an involution. Let $\Phi$ denote the
  conjugacy by $\tau$: that is, $\Phi(\sig)=\tau\sig\tau$. Of course,
  $\Phi(\id)=\id$ and $\Phi(\si \si')= \Phi(\si)\Phi(\si')$. Also, $\Phi$ has a
  simple description in terms of the \emm diagram, of $\sig$, in which
  $\sig(i)$ is plotted against $i$: the diagram of $\Phi(\sig)$ is
  obtained by applying a rotation of 180 degrees to the diagram of
  $\sig$. In particular, with $s_i=(i,i+1)$, one has
  $\Phi(s_i)=s_{m-1-i}$ for $0\le i <m$.  These properties imply that the   sequence of random
  \ps\ $(\Phi(\pi^{(0)}), \Phi(\pi^{(1)}), \Phi(\pi^{(2)}), \ldots)$
  follows the same law as the original Markov chain $(\pi^{(0)},
  \pi^{(1)}, \pi^{(2)}, \ldots)$. In particular, for $0\le i \le j <m$,
$$
p_{i,j}^{(n)}= \PP( \Phi(\pi^{(n)})_i > \Phi(\pi^{(n)})_{j+1})
= \PP ( \pi^{(n)}_{m-j-1} > \pi^{(n)}_{m-i}) 
= p^{(n)}_{m-j-1, m-i-1}.
$$
This is equivalent to the first statement of the lemma. The second
follows by specialization.
\end{proof}

The next ingredient in our solution is the ``obstinate'' kernel method
of~\cite{bousquet-motifs,Bous05,bousquet-mishna,bousquet-xin,Rensburg-Prellberg-Rechni,Mishna-Rechni,xin-zhang}.   
The \emm kernel, of the functional equation~\eqref{eqP} is
the coefficient of $P(u,v)$, namely 
$$
K(u,v)= 1-t +\frac t m (4-u-\bu-v-\bv).
$$
Let $(U,V)$ be a pair of Laurent series in $t$ that cancel the
kernel: $K(U,V)=0$. The series $P(U,V)$ is well-defined (because
$P(u,v)$ is a polynomial in $u$ and $v$). Setting $u=U$ and $v=V$
in~\eqref{eqP} cancels the left-hand side,  and thus the
right-hand side. That is, denoting as usually $\bU=1/U$ and $\bV=1/V$,
$$
 \frac {\bV^{m-1}(1-U^m V^m)}{(1-U V)(\bU -1)(V-1)(1-t)}
- \frac{\bV^{m-1} P_\ell(V)}{V-1} -\frac{P_t(U)}{\bU -1} 
-\frac{\bV^{m-1}(U+\bV)}{(\bU -1)(V-1)} P_d(UV) = 0
$$
provided $U\not=1$ and $V\not =1$.
Let us now exploit the symmetries of the kernel: obviously, $K(u,v)$
is invariant by the transformations $u\mapsto \bu$ and $v\mapsto
\bv$. Hence the pairs $(\bU, V), (\bU, \bV)$ and $(U, \bV)$ also
cancel $K$, and it follows that 
$$
 \frac {\bV^{m-1}(1-\bU^m V^m)}{(1-\bU V)(U -1)(V-1)(1-t)}
- \frac{\bV^{m-1} P_\ell(V)}{V-1} -\frac{P_t(\bU)}{U -1} 
-\frac{\bV^{m-1}(\bU+\bV)}{(U -1)(V-1)} P_d(\bU V) = 0,
$$
$$
 \frac {V^{m-1}(1-\bU^m \bV^m)}{(1-\bU \bV)(U -1)(\bV-1)(1-t)}
- \frac{V^{m-1} P_\ell(\bV)}{\bV-1} -\frac{P_t(\bU)}{U -1} 
-\frac{V^{m-1}(\bU+V)}{(U -1)(\bV-1)} P_d(\bU\bV) = 0,
$$
$$
 \frac {V^{m-1}(1-U^m \bV^m)}{(1-U \bV)(\bU -1)(\bV-1)(1-t)}
- \frac{V^{m-1} P_\ell(\bV)}{\bV-1} -\frac{P_t(U)}{\bU -1} 
-\frac{V^{m-1}(U+V)}{(\bU -1)(\bV-1)} P_d(U\bV) = 0.
$$

\smallskip
Let us form the alternating sum of the four previous equations: all
occurrences of $P_\ell$ and $P_t$ vanish, and,
using~\eqref{symm-diag}, one obtains, after multiplying by $U^m V^m (1-U)(1-V)$:
\begin{multline}
  \left( UV+1 \right) 
\left( {U}^{m+1}+{V}^{m+1} \right) 
P_d \left( UV \right) 
+  {V}^{m-1}\left( U+V \right) \left( 1+{U}^{m+1}{V}^{m+1} \right)
P_d \left(U\bV  \right) 
\label{Pd-2}\\
=\frac {UV} {1-t}
\left(
{\frac { \left( 1-{U}^{m}{V}^{m} \right)  \left( {U}^{m}+{V}^{m}
    \right) }
{1-UV}}
-{\frac { \left( {U}^{m}-{V}^{m} \right)  \left( 1+{U}^{m}V^{m} \right) }{V-U}}
\right) 
\end{multline}
as soon as $K(U,V)=0$, $U\not=1$ and $V\not =1$.

This equation involves only one unknown series, namely $P_d$. So far, the series
$U$ and $V$ are coupled by a single condition, $K(U,V)=0$.  Let $q$ be a
complex root of $q^{m+1}=-1$, and
let us add a second constraint on the pair $(U,V)$ by
requiring that $U=q V$. That is,  $V$ must be a root of 
\beq\label{Ker-def-q}
K(qV,V)=1-t + \frac t m \left( 4- (1+q)(V+\bq \bV)\right)=0
\eeq
where $\bq =1/q$ and $\bV =1/V$. 
We further assume $q\not = -1$ (otherwise $K(qv,v)$ is independent of
$v$). Then $V\not =1$, $U=q V\not =1$, and the first
term of~\eqref{Pd-2} vanishes.  We thus obtain an explicit
expression  of $P_d(q)$, which we write below in terms of $V$ and its
conjugate root  $V'=\bq \bV$. 

\begin{Lemma}\label{lem:Pd1}
  Let $q\not = -1$ satisfy $q^{m+1}=-1$, and let $V\equiv V(t)$ and $V'\equiv V'(t)$ be the two roots of~\eqref{Ker-def-q}.
Then $V V'=\bq=1/q$ and
$$
P_d(q)= \frac \bq {(1-t)(V^{m+1}+V'^{m+1})}
\left(
\frac {V^m+V'^m}{1-q} + \frac {1-\bq}{1+q}\, \frac{V^m-V'^m}{V-V'}
\right).
$$
\end{Lemma}
 The series $V$ and $V'$ are algebraic (of degree 2) over
 $\cs(t)$. But their symmetric  functions are rational, and thus
 $P_d(q)$ is rational too, as expected.  The following lemma gives an explicit
rational expression of $P_d(q)$.
\begin{Lemma}\label{lem:Pd2}
  Let $q\not = -1$ satisfy $q^{m+1}=-1$. Assume
  $q=q_k:=e^ {i \frac{2k+1}{m+1}\pi}$.
Then
$$
P_d(q_k) 
= \frac {2it}{q_k(1-t) m(m+1) s_k} 
\sum_{j=0}^m  \frac{(c_k+c_j)(1-c_kc_j)}{1-tx_{jk}}
$$
with 
$$
c_j= \cos \frac{(2j+1)\pi}{2m+2}, \quad 
s_j= \sin \frac{(2j+1)\pi}{2m+2} 
\quad \hbox{and} \quad x_{jk}= 1-\frac 4 m (1-c_jc_k).
$$
Equivalently,
$$
P_d(q_k)= \frac {i c_k} {2 q_k(1-t)s_k}
-   \frac {i}{2q_k (m+1) s_k} 
\sum_{j=0}^m  
\frac{c_k+c_j}{1-tx_{jk}}.
$$
\end{Lemma}
\begin{proof}
We will establish a closed form expression for the coefficient of
  $t^n$ in  $(1-t)P_d(q)$, which is clearly equivalent to the first
expression of $P_d(q)$ given above:
for $n\ge 1$,
\begin{eqnarray}
  a_n&:=& [t^n] (1-t)P_d(q) \nonumber
\\ &=&
 \frac {2i}{q m(m+1) s_k} 
\sum_{j=0}^m {(c_k+c_j)(1-c_kc_j)}\ {x_{jk}}^ {n-1}.
\label{an-expr}
\end{eqnarray}
In order to obtain this expression, we begin with applying Cauchy's
formula to the expression of $(1-t)P_d(q)$ given in Lemma~\ref{lem:Pd1}. Let $V$
be the root of~\eqref{Ker-def-q} that vanishes at $t=0$. 
(The other root $V'=\bq \bV$ has a term in $O(1/t)$ in its expansion.) 
Then
$$
a_n= \frac 1 {2i\pi} \int _\circlearrowleft R(V) \frac {dt}{t^{n+1}}
$$
where the integral is taken along  a small circle around the origin,
in counterclockwise direction,
$$
R(v)=  \frac \bq {v^{m+1}+(\bq\bv)^{m+1}}
\left(
\frac {v^m+(\bq\bv)^m}{1-q} + \frac {1-\bq}{1+q}\, \frac{v^m-(\bq\bv)^m}{v-\bq\bv}
\right),
$$
and $\bv= 1/v$. Obviously, $R(v)=R(\bq \bv)$.
By~\eqref{Ker-def-q}, 
$$
\frac 1 t = 1 -\frac 4 m +\frac{1+q} m (V+\bq \bV)
$$
so that
$$
\frac {dt } {t^2} = -\frac{1+q} m (V-\bq \bV) \frac {dV} V.
$$
Thus the integral expression of $a_n$ reads
$$
a_n= \frac 1 {2i\pi} \int_\circlearrowleft   S(v) \frac {dv}{v}
$$
where the integral is taken along  a small circle around the origin,
in counterclockwise direction, and
$$
S(v)=-\frac{1+q} m (v-\bq \bv) \left(1 -\frac 4 m +\frac{1+q} m (v+\bq
\bv)\right)^{n-1} R(v)
$$
 satisfies $S(\bq\bv)=-S(v)$.
Note that the only possible poles of $S(v)/v$ are  $0$ and the $(2m+2)$th
roots of unity (because $q^{m+1}=-1$). Thus $a_n$ is simply the
residue of $S(v)/v$ at $v=0$.
Performing the change of variables $v=\bq \bw$, with $\bw=1/w$, gives
$$
a_n= \frac 1 {2i\pi} \int_{ {\circlearrowright}}   S(w) \frac {dw}{w}
$$
where the integral is now taken along  a \emm large, circle around the origin,
in clockwise direction. This integral thus collects (up to a sign)
\emm all, residues of $S(v)/v$.
The residue formula thus gives:
\begin{eqnarray}
  2a_n &=& \frac 1 {2i\pi} \int_{ {\circlearrowright}}   S(w) \frac {dw}{w} +
 \frac 1 {2i\pi} \int_\circlearrowleft  S(v) \frac {dv}{v}
\nonumber \\
&=&
-\sum_{v^{2m+2} =1} \frac 1 v \, \Res_{v}(S)
\nonumber \\ 
&=&
\frac 1 {2m+2} \sum_{v^{2m+2} =1} P(v) \label{an-expr-int}
\end{eqnarray}
where 
\begin{multline*}
P(w)= (1-w^{2m+2}) S(w)=
\\
 \frac{1+\bq} m w^{m+1}(w-\bq \bw)
\left(
\frac {w^m+(\bq\bw)^m}{1-q} + \frac {1-\bq}{1+q}\,
\frac{w^m-(\bq\bw)^m}
{w-\bq\bw}
\right)
 \left(1 -\frac 4 m +\frac{1+q} m (w+\bq
\bw)\right)^{n-1} 
.
\end{multline*}
Take $v= e^{i\al}$ with $\al= \ell \pi/(m+1)$ and $0\le \ell
<2m+2$ and recall that $q=e^ {i \theta}$ with $\theta=
\frac{(2k+1)\pi}{m+1}$. Then 
\begin{eqnarray*}
  {1+\bq} &=& 2e^{-i\theta/2} \cos (\theta/2),
 \\
{1-q} &= & -2i e^{i\theta/2} \sin (\theta/2),
\\
 {1-\bq} &=& 2ie^{-i\theta/2} \sin (\theta/2),
\\
{1+q} &= & 2 e^{i\theta/2} \cos(\theta/2),
\\
 v^{m+1} &= &(-1)^\ell,
\\
v-\bq \bv &= &2i e^{-i\theta/2}\sin (\al+\theta/2),
\\
v+\bq \bv &= &2 e^{-i\theta/2}\cos(\al+\theta/2),
\\
v^m+(\bq\bv)^m &= & 
 2 i(-1)^{\ell+1} e^{i\theta/2}\sin (\al+\theta/2),
\\
v^m-(\bq\bv)^m
 &= & 
 2 (-1)^{\ell} e^{i\theta/2}\cos (\al+\theta/2).
\end{eqnarray*}
Putting these identities together, one obtains
\begin{multline*}
  P(v)= \frac{4i}{qm \sin (\theta/2)} \left(\cos (\theta/2)
\sin^2(\al+\theta/2)
+\sin^2 (\theta/2) \cos (\al+\theta/2)\right)
\\
\times \left(1-\frac 4 m + \frac 4 m \cos( \theta/2) \cos (\al+\theta/2) \right)^{n-1},
\end{multline*}
or,  with the notation of the lemma,
$$
P(v)= \frac{4i}{qm s_k} (c_j+c_k)(1-c_jc_k) \ {x_{jk}}^{n-1}
$$
with $j=k+\ell$.
Returning  to~\eqref{an-expr-int} now gives 
$$
a_n=
 \frac {i}{q m(m+1) s_k} \sum_{j=-m-1}^m
       {(c_k+c_j)(1-c_kc_j)} \ {x_{jk}}^{n-1},
$$
which is equivalent to~\eqref{an-expr}, upon noting that $c_j=c_{-j-1}$.
The first expression of $P_d(q)$ given in the lemma follows.

For the second expression, we simply perform a partial fraction
expansion in the variable $t$, based on 
\beq\label{des}
\frac t{(1-t)(1-xt)}= \frac 1 {1-x} \left( \frac 1 {1-t} - \frac 1
      {1-xt}\right),
\eeq
and use 
$$
\sum _{j=0}^m c_j=0
$$
(this identity follows for instance from $c_{m-j}=-c_j$).
\end{proof}

Recall that $P_d(u)$ is a polynomial in $u$ of degree $m-1$. The above
lemma gives its values at $m$, or even $m+1$, distinct points. Thus
 $P_d(u)$ is   completely determined by these values, and we can
recover it by interpolation. We use a version of Lagrange's
interpolation  that is well-suited to symmetric polynomials.

\begin{Lemma}\label{lem:lagrange}
  Let $P(u)$ be a  polynomial of degree $m-1$ with
  coefficients in some field containing $\cs$. Assume $P(u)$ is
  symmetric.
That is, 
$P(u)=u^{m-1} P(\bu)$, with $\bu=1/u$. Let $\ell=\lfloor \frac {m-1}
  2\rfloor$, and let $q_0, \ldots, q_\ell$ be distinct elements of
  $\cs$ such that $q_jq_k\not = 1$ for all $k$ and $j$. Then
\beq
\label{Pu-int}
P(u)= (1+u)^{\chi_{m,0}} \sum_{k=0}^\ell \frac{P(q_k)}{(1+q_k)^{\chi_{m,0}} } \prod_{j\not = k}
\frac{(u-q_j)(u-1/q_j)}{(q_k-q_j)(q_k-1/q_j)},
\eeq
where $\chi_{m,0}=1$ if $m$ is even, and $0$ otherwise. 

When $q_k= e^{i\theta_k}$ with  $\theta_k= \frac{(2k+1)\pi}{m+1}$,
this can be rewritten as
$$
P(u)= -\frac{2i}{m+1} (u^{m+1}+1) \sum_{k=0}^\ell P(q_k)\, \frac{q_k
  \sin \theta_k}{u^2-2u \cos \theta_k+1}.
$$
In particular, with the notation~\eqref{notation},
$$
P(1)= -\frac{2i}{m+1} \sum_{k=0}^\ell P(q_k)\, \frac{q_k  c_k}{s_k}.
$$
\end{Lemma}
\begin{proof}
  For the first part, it suffices to observe that the expression given
  on the right-hand side of~\eqref{Pu-int} has degree $m-1$, and takes the same values as
  $P(u)$ at the $2\ell+2$ distinct points $q_0, \ldots, q_\ell, 1/q_0, \ldots,
  1/q_\ell$, and also at $-1$ if $m$ is even. This gives a total of
  $m+1$ correct values, which is more than enough to determine a polynomial of
  degree $m-1$.

For the second part, we observe that $q_0, \ldots, q_\ell, 1/q_0, \ldots,
  1/q_\ell$, together with $-1$ if $m$ is even, are the $m+1$ roots of
  $u^{m+1}+1$. Hence
$$
(1+u)^{\chi_{m,0}} \prod_{j\not = k} {(u-q_j)(u-1/q_j)}= 
\frac{u^{m+1}+1}{(u-q_k)(u-1/q_k)},
$$
which, in the limit $u\rightarrow q_k$, gives
$$
(1+q_k)^{\chi_{m,0}} \prod_{j\not = k} {(q_k-q_j)(q_k-1/q_j)}= 
- \frac{m+1}{q_k(q_k-1/q_k)}.
$$
The second result follows. The third one is obtained by setting $u=1$,
and noticing that $\sin \theta_k=2 c_k s_k$, while $1-\cos \theta_k=2s_k^2$.
\end{proof}

We can finally combine this interpolation formula with Lemma~\ref{lem:Pd2} to obtain
an explicit expression of $P_d(u)$. As we are mostly interested in
$P_d(1)$ (see~\eqref{Im-Pd}), we give only this value.

\begin{Proposition}
\label{prop:Pd1}
Let  $\ell=\lfloor \frac {m-1}  2\rfloor$, and adopt the
notation~\eqref{notation}.
  Then the series $P_d(u)$ defined by~\eqref{Pd-def} satisfies:
\begin{eqnarray*}
P_d(1)&=&
 \frac m {2(1-t)}
- \frac 1{(m+1)^2} \sum _{k=0}^\ell  \sum_{j=0}^m \frac{c_k}{s_k^2}
\frac{c_j+c_k}{1-tx_{jk}}
\\&=&
 \frac m {2(1-t)}
- \frac 1{4(m+1)^2} \sum _{k=0}^m \sum_{j=0}^m \frac{(c_j+c_k)^2(1-c_jc_k)}{s_j^2s_k^2}
\frac{1}{1-tx_{jk}}.
\end  {eqnarray*}
\end{Proposition}
\begin{proof}
We apply Lemma~\ref{lem:lagrange} to the second expression of
$P_d(q_k)$ obtained in Lemma~\ref{lem:Pd2}. This gives the first
expression of $P_d(1)$ above, provided
\beq\label{sumcs}
\sum_{k=0}^\ell\frac{c_k^2}{s_k^2}= \frac{m(m+1)}2.
\eeq
The latter identity is obtained by applying Lemma~\ref{lem:lagrange}
to $P(u)=1+u+\cdots + u^{m-1}= (1-u^m)/(1-u)$.

We now seek a symmetric formula in $j$ and $k$. Using $c_{m-j}=-c_j$,
$c_{m-k}=-c_k$,  $s_{m-k}=s_k$ and $x_{j,k}=x_{m-j,m-k}$, we write
\begin{eqnarray*}
 \sum _{k=0}^\ell  \sum_{j=0}^m \frac{c_k}{s_k^2}
\frac{c_j+c_k}{1-tx_{jk}}
&=&
 \sum _{k=0}^\ell  \sum_{j=0}^m \frac{c_{m-k}}{s_{m-k}^2}
\frac{c_{m-j}+c_{m-k}} {1-tx_{m-j,m-k}}
\\&=&
 \sum _{k=m-\ell}^m \, \sum_{j=0}^m\, \frac{c_k}{s_k^2}
\frac{c_j+c_k} {1-tx_{jk}}
\\&=&
\frac 1 2  \sum _{k=0}^m  \sum_{j=0}^m \frac{c_k}{s_k^2}
\frac{c_j+c_k} {1-tx_{jk}}
\\&=&
\frac 1 4  \sum _{k=0}^m  \sum_{j=0}^m \left(\frac{c_k}{s_k^2}+\frac{c_j}{s_j^2}\right)
\frac{c_j+c_k} {1-tx_{jk}}
\end{eqnarray*}
and this gives the second expression of $P_d(1)$.
\end{proof}

\medskip
\noindent
{\em Proof of Theorem~{\rm\ref{thm:main}.}}
Let us now return to the \gf\ $\II_m(t)$ whose coefficients give the
expected number of inversions. It is related to $P_d(1)$
by~\eqref{Im-Pd}. Theorem~\ref{thm:main} is obtained by combining the second
expression of Proposition~\ref{prop:Pd1}, 
 a partial fraction expansion in $t$ (based on~\eqref{des}),
and finally the identity
\beq\label{cs-id}
\sum_{j,k=0}^m \frac{(c_j+c_k)^2}{s_j^2 s_k^2}= 2m(m+1)^3.
\eeq
To prove this identity, we write
\begin{eqnarray*}
  \sum_{j,k=0}^m \frac{(c_j+c_k)^2}{s_j^2 s_k^2}
&=& \sum_{j,k=0}^m \frac{c_j^2+c_k^2}{s_j^2 s_k^2}\quad \quad \quad  (\hbox{as }
  c_{m-j}=-c_j
\hbox{ and } s_{m-j}=s_j)\\
&=&
2\left ( \sum_{j=0}^m \frac{c_j^2}{s_j^2}\right) \left ( \sum_{k=0}^m\frac 1 {s_k^2}\right)
\\
&=&
2\left (\sum_{j=0}^m \frac{c_j^2}{s_j^2}\right) \sum_{k=0}^m\left ( 1 +\frac{c_k^2}  {s_k^2}\right)
\end{eqnarray*}
and complete the proof thanks to~\eqref{sumcs}.

To obtain the  expression~\eqref{I-ser-2} of $\II_m(t)$, we write
\begin{eqnarray*}
\sum_{j,k=0}^m \frac{(c_j+c_k)^2}{s_j^2 s_k^2} \frac 1{1-tx_{jk}}
&=&
\frac 1 2 \sum_{j,k=0}^m\left(\frac{c_j+c_k}{(1-c_j)(1-c_k)}
-\frac{c_j+c_k}{(1+c_j)(1+c_k)}
 \right)\frac 1{1-tx_{jk}}
\\&=&
\sum_{j,k=0}^m\frac{c_j+c_k}{(1-c_j)(1-c_k)}\frac 1{1-tx_{jk}}.
\end{eqnarray*}
The latter identity follows from replacing $j$ by $m-j$ and $k$ by
$m-k$.

Let us now extract the coefficient of $t^0$ in~\eqref{I-ser-2}. This
gives
\beq\label{iddouble}
0=\frac{m(m+1)}{4} 
- \frac 1 {8(m+1)^2}\sum_{j,k=0}^m\frac{c_j+c_k}{(1-c_j)(1-c_k)},
\eeq
and  the  expression~\eqref{I-ser-3} of $\II_m(t)$ follows.

\section{Small times: linear and before}
\label{sec:small}
When $m$ is fixed and $n \rightarrow \infty$, the asymptotic behaviour
of the numbers $\II_{m,n}$  is easily derived from
Theorem~\ref{thm:main}, as sketched just after the statement of this
theorem. For $m\ge 3$,
$$
I_{m,n}= \frac{m(m+1)}4 + O({x_{00}}^n)
$$
where $x_{00}= 1-\frac 4 m \sin ^2\! \frac \pi{2m+2}$.

In this section and the next two ones, we consider the case where
$n\equiv n_m$ depends on $m$, and $m\rightarrow \infty$. 
As in~\cite{b-d-adjacent}, three main regimes appear: linear ($n_m=\Theta(m)$)),
cubic  ($n_m=\Theta(m^3)$) and intermediate ($m\ll n_m \ll m^3$). This
can be partly explained using the following simple bounds.

\begin{Lemma}\label{lem:bounds}
  For $m\ge 3$ and $n\ge 0$, there holds
$$
\frac{m(m+1)}4 (1-{x_{00}}^n) \le \II_{m,n} \le \frac{m(m+1)}4-
\frac{c_0^2}{2(m+1)^2s_0^4}\, {x_{00}}^n. 
$$
In particular, if $n\equiv n_m$ and  $m\rightarrow \infty$,
$$
\II_{m,n} = \frac{m(m+1)}4 -\Theta(m ^2 {x_{00}}^n).
$$
\end{Lemma}
\begin{proof}
These inequalities follow from the expression of $\II_{m,n}$ given in
Theorem~\ref{thm:main}. The upper bound is obtained by retaining only,
in the sum over $j$ and $k$, the term obtained for $j=k=0$. The lower
bound  follows from $|x_{jk}|\le x_{00}$ and~\eqref{cs-id}.
\end{proof}
Observe that $x_{00}=1-O(1/m^3)$. The lower bound on $\II_{m,n}$ already shows that
if $n\gg m^3$, then ${x_{00}}^n=o(1)$ and $\II_{m,n} \sim
\frac{m(m+1)}4$. Also, if $n\sim \kappa m^3$ for some $\kappa >0$,
then ${x_{00}}^n\sim \alpha$ for some $ \alpha \in (0,1)$: then $\II_{m,n} $ is
still quadratic in $m$, but the upper bound shows that the ratio
$\II_{m,n}/m^2 $ will be less than $1/4$. The other regimes correspond
to cases where $n\ll m^3$.

\medskip

This
section is devoted to the linear (and sub-linear) regime. We first
state our results, and then comment on their meaning.

\begin{Proposition}\label{prop:linear}
Assume $n\equiv n_m=o(m)$. Then
$$
\frac{\II_{m,n}} n=1+ O(n/m).
$$
  Assume $n\equiv n_m=\Theta(m)$. That is, $\kappa_1 m \le n_m\le \kappa_2 m$ for two positive constants $\kappa_1$ and $\kappa_2$. 
Then
$$
\frac{\II_{m,n}}n = f(n/m)+ O(1/m)
$$
where
\begin{eqnarray*}
  f(\kappa)&=&
\frac 1 {2 \pi \kappa} \int_0^\infty
\frac{1-\exp(-8\kappa t^2/(1+t^2))}{t^2(1+t^2)} dt
\\
&=& \sum_{j\ge 0} (-1)^{j} \frac {(2j)!}{j!(j+1)!^2}(2\kappa)^j
.
\end{eqnarray*}
\end{Proposition}
\noindent{\bf Comment.}
When $n_m$ is sub-linear, the inversion number equals its largest
possible value, $n$, with high probability. For instance,
$\II_{1000,10}\simeq 9.9$.
When $n_m$ grows linearly with $m$, the expected inversion number is still
linear in $n$, but with a ratio $f(\kappa)$, where
$\kappa=n/m$. This ratio decreases from $f(0)=1$ to $f(\infty)=0$ as
 $\kappa$ increases from $0$ to $\infty$. The fact  that $f(0)=1$ is
consistent with the sub-linear result.

Note that for a related continuous time chain, with inversion number
$D_m^t$,  it has been proved that when $t=\kappa m$, the random
variable $D_m^t/t$ converges in probability to a function of $t$ described in
probabilistic terms~\cite{b-d-adjacent}.
 \begin{proof}
  The starting point of both results is the following expression of
  $\II_{m,n}$, which  corresponds to~\eqref{I-ser-3}:  
\beq\label{Imn-expr-3}
\II_{m,n}=\frac 1 {8(m+1)^2}\sum_{k,j=0}^m
\frac{c_j+c_k}{(1-c_j)(1-c_k)} \left( 1-\left(1- \frac 4 m (1-c_jc_k)\right)^n
\right).
\eeq
$\bullet$ Assume $n_m=o(m)$. Then
$$
\left(1- \frac 4 m (1-c_jc_k)\right)^n = 1- \frac {4n} m (1-c_jc_k)
+(1-c_jc_k)^2\,  O\!\left( \frac{n^2}{m^2}\right),
$$
uniformly in $j$ and $k$.
Thus
$$
\II_{m,n}=\frac n {2m(m+1)^2}\sum_{k,j=0}^m
\frac{(c_j+c_k) (1-c_jc_k)}{(1-c_j)(1-c_k)} 
+
\frac 1 {4(m+1)^2}\sum_{k,j=0}^m
\frac{(1-c_jc_k)^2}{(1-c_j)(1-c_k)} \,
O\!\left( \frac{n^2}{m^2}\right) 
$$
(we have bounded $|c_j+c_k|$ by 2)
and the result follows using
$$
\sum_{k,j=0}^m
\frac{(c_j+c_k) (1-c_jc_k)}{(1-c_j)(1-c_k)} = 2m(m+1)^2
$$
and
\beq
\label{iddouble-2}
\sum_{k,j=0}^m
\frac{(1-c_jc_k)^2}{(1-c_j)(1-c_k)}= (2m+1)(m+1)^2.
 \eeq
To prove these two identities, one may start from the following
``basic'' identity:
\beq
\label{basic-sum}
\sum_{j=0}^m \frac 1{1-c_j}= (m+1)^2,
\eeq
which follows for instance from~\eqref{iddouble}.

\medskip
\noindent $\bullet$ Assume now $n_m=\Theta(m)$ and denote $\kappa=n/m$. Then
$$
\left(1- \frac 4 m (1-c_jc_k)\right)^n=\exp\left( -{4 \kappa}
(1-c_jc_k)\right) \left(1+ (1-c_jc_k)^2 O(n/m^2)\right).
$$
Thus
\begin{multline*}
  \II_{m,n}= \frac 1 {8(m+1)^2}\sum_{k,j=0}^m
\frac{c_j+c_k}{(1-c_j)(1-c_k)} \left( 1-\exp\left( -{4 \kappa}
(1-c_jc_k)\right)\right)\\
+ \frac 1 {8(m+1)^2}\sum_{k,j=0}^m
\frac{(c_j+c_k)(1-c_jc_k)^2 }{(1-c_j)(1-c_k)} 
\exp\left( -{4 \kappa}(1-c_jc_k)\right)O(n/m^2)
.
\end{multline*}
By~\eqref{iddouble-2}, the absolute value of the second term above is bounded by
$$
\frac 1 {4(m+1)^2}\sum_{k,j=0}^m
\frac{(1-c_jc_k)^2 }{(1-c_j)(1-c_k)} 
O(n/m^2)
= O(n/m)=O(1).
$$
Recall that $c_j= \cos \frac{(2j+1)\pi}{2m+2}$. Hence the first term
in the expression of $\II_{m,n}$  looks very much like a (double) Riemann sum,
but one must be careful, as the integral
$$
\int_0^\pi\int_0^\pi \frac{\cos x + \cos y}{(1-\cos x)(1-\cos y )} 
\left( 1-\exp\left( -{4 \kappa}(1-\cos x \cos y)\right) \right)dx\, dy
$$
 diverges (the integrand behaves like $1/x$ around $x=0$, and like
$1/y$ around $y=0$). Let us thus write
\begin{multline*}
  1-\exp\left( -{4 \kappa}(1-c_jc_k)\right)=
\left[1-\exp\left( -{4 \kappa}(1-c_j)\right)
\right]
+
\left[1-\exp\left( -{4 \kappa}(1-c_k)\right)
\right]
\\-
\left[1-\exp\left( -{4 \kappa}(1-c_j)\right)
       -\exp\left( -{4 \kappa}(1-c_k)\right)
+\exp\left( -{4 \kappa}(1-c_jc_k)\right)
\right],
\end{multline*}
so that
$$
  \II_{m,n}= 2S_1-S_2+ O(1)
$$
with
$$
S_1=\frac 1 {8(m+1)^2}\sum_{k,j=0}^m
\frac{c_j+c_k}{(1-c_j)(1-c_k)} \left( 1-\exp\left( -{4
  \kappa}(1-c_j)\right)\right)
$$
and
\begin{multline*}
S_2=\frac 1 {8(m+1)^2}\sum_{k,j=0}^m
\frac{c_j+c_k}{(1-c_j)(1-c_k)} 
\\
\times \left[1
-\exp\left( -{4 \kappa}(1-c_j)\right)
-\exp\left( -{4 \kappa}(1-c_k)\right)
+\exp\left( -{4 \kappa}(1-c_jc_k)\right)
\right].
\end{multline*}
The first sum reads
\begin{eqnarray*}
  S_1&=&\frac 1 {8(m+1)^2}\sum_{j=0}^m\frac { 1-\exp\left( -{4
  \kappa}(1-c_j)\right)} {1-c_j}
\sum_{k=0}^m
\frac{c_j+c_k}{1-c_k} \\
&=&\frac 1 {8}\sum_{j=0}^m\frac { 1-\exp\left( -{4
  \kappa}(1-c_j)\right)} {1-c_j} \left( 1+ c_j\right)
-\frac 1 {8(m+1)}\sum_{j=0}^m\frac { 1-\exp\left( -{4
  \kappa}(1-c_j)\right)} {1-c_j}
,
\end{eqnarray*}
as
$$
\quad \sum_{k=0}^m \frac{1}{1-c_k}= (m+1)^2 \quad \hbox{ and }
\quad \sum_{k=0}^m \frac{c_k}{1-c_k}= m(m+1)
$$
(see~\eqref{basic-sum}).
Both terms in $S_1$ are now \emm bona fide, Riemann sums. More precisely,
$$
 \frac\pi{m+1}\sum_{j=0}^m\frac { 1-\exp\left( -{4
  \kappa}(1-c_j)\right)} {1-c_j} \left( 1+ c_j \right)
=\int_0 ^\pi
\frac{1-\exp(-4\kappa(1-\cos x))}{1-\cos x}(1+\cos x)
dx 
+O(1/m)
$$
uniformly in $\kappa \in
[\kappa_1, \kappa_2]$, and similarly
$$
 \frac1{m+1}\sum_{j=0}^m\frac { 1-\exp\left( -{4
  \kappa}(1-c_j)\right)} {1-c_j} 
=O(1).
$$
 Thus
$$
S_1=\frac m {8\pi}
\int_0 ^\pi
\frac{1-\exp(-4\kappa(1-\cos x))}{1-\cos x}(1+\cos x)
dx + O(1).
$$
Similarly, $S_2$ is a (double) Riemann sum associated with a converging
integral, and is $O(1)$. 
Thus
$$
\II_{m,n}=2S_1+O(1)
=
\frac n {4\pi \kappa}
\int_0 ^\pi
\frac{1-\exp(-4\kappa(1-\cos x))}{1-\cos x}(1+\cos x)\,dx +O(1).
$$
 The integral can
be rewritten  by setting $t=\tan (x/2)$, and this gives the second
result of the proposition, with the integral expression of
$f(\kappa)$. The expansion of $f(\kappa)$ in $\kappa$ is then routine. 
\end{proof}

\section{Large times: cubic and beyond}
\label{sec:large}
\begin{Proposition}\label{prop:cubic}
If $n\equiv n_m \gg m^3$, then 
$$
\frac{\II_{m,n}}{m^2} \rightarrow \frac 1 4.
$$
 Assume 
$n\equiv n_m=\Theta(m^3)$. That is, $\kappa_1 m^3 \le n_m\le \kappa_2
 m^3$ for two positive constants $\kappa_1$ and $\kappa_2$. 
Then
$$
\frac{\II_{m,n}}{m^2} \sim g(n/m^3)
$$
where
$$
g(\kappa)= \frac 1 4 
-\frac {16}{\pi^4}\left(\sum_{j\ge0} \frac{e^{-\kappa\pi^2 (2j+1)^2/2}}{(2j+1)^2}
\right)^2.
$$
\end{Proposition}
\noindent{\bf Comment.}
When $n \gg m^3$, the inversion number equals, at least at first
order, its ``limit''
value $\frac{m(m+1)}4 $, which is the average number of inversions in
a \p\ of $\Sn_{m+1}$ taken uniformly at random.
When $n=\Theta(m^3)$,  the inversion number is still quadratic in $m$, 
but with a ratio $g(\kappa)$, where $\kappa=n/m^3$. This ratio increases from 0 to $1/4$ as $\kappa$ goes from $0$ to
$\infty$, which makes this result consistent with the super-cubic
regime.
Note that for a related continuous time chain, with inversion number
$D_m^t$,  it has been proved that when $t=\kappa m^3$, the random
variable $D_m^t/m^2$ converges in probability to a function of $t$  described in
probabilistic terms~\cite{b-d-adjacent}.

As said above, $\II_{m,n}\sim m^2/4$ as soon as  $n \gg m^3$. 
Whether the next term in  the expansion of
$\II_{m,n}$ is exact, that is, equal to $m/4$, depends on how $n$ compares to $m^3\log m$, as shown by  the
following proposition. In particular, it clarifies when
$\II_{m,n}=\frac{m(m+1)}4+ o(m)$: this happens beyond $n\sim 1/\pi^2 m^3
\log m $. 

\begin{Proposition}\label{prop:cubic-log}
Let $c>0$, and $n\equiv n_m \sim c m^3\log m$. For all $\epsilon >0$,
$$ 
  \frac{m(m+1)}4- O(m^{2+\epsilon -c\pi^2})
= \II_{m,n} \le  \frac{m(m+1)}4- \Theta(m^{2-\epsilon -c\pi^2}).
$$
Thus if $c<1/\pi^2$, there exists $\gamma>0$ such that
$$
 \II_{m,n} \le \frac{m(m+1)}4- \Theta(m^{1+\gamma}),
$$
 while if $c>1/\pi^2$, there exists $\gamma>0$ such that
$$
 \II_{m,n} = \frac{m(m+1)}4- O(m^{1-\gamma}).
$$
For the critical value $c=1/\pi^2$, the following refined estimate
holds: if $n\equiv n_m \sim 1/\pi^2  m^3\log m + \alpha m^3 + o(m^3)$, then
$$ 
\II_{m,n}=\frac{m(m+1)}4 - \frac{16 m }{\pi^4} e^{-\alpha  \pi^2} (1+o(1)).
$$
\end{Proposition}

\noindent{\em Proof of Proposition~\rm{\ref{prop:cubic}}.}
   The first result is a direct consequence of Lemma~\ref{lem:bounds},
   given that $x_{00}=1-O(1/m^3)$.

Assume now that $n_m=\Theta(m^3)$.
We start from the  following expression of
  $\II_{m,n}$, which  corresponds to~\eqref{I-ser-2}: 
\beq\label{expr-Imn-5}
\II_{m,n}=\frac{m(m+1)} {4}-\frac 1 {8(m+1)^2}\sum_{k,j=0}^m
\frac{c_j+c_k}{(1-c_j)(1-c_k)} \left(1- \frac 4 m (1-c_jc_k)\right)^n.
\eeq
Let $M\equiv M_m$  be an integer sequence that tends to infinity
in such a way $M_m=o(\sqrt m)$.
We split the sum over $j$ and $k$ into two parts: $j \le M$, $k\le M$
in the first part, $j>M$ or $k>M$ in the second part.
Let us prove that  the second part can be neglected. We have:
\begin{eqnarray*}
  \left| \SUM_2 \right|&:=& \left|
\sum_{j \hbox{ \small or } k>M }
\frac{c_j+c_k}{(1-c_j)(1-c_k)} {x_{jk}}^n\right|
\\
&\le&
 \sum_{j \hbox{ \small or } k>M }\frac{2}{(1-c_j)(1-c_k)}
\ \le \ 
4  \sum_{k =0}^m \frac{1}{1-c_k}
\sum_{j=M+1}^m \frac{1}{1-c_j}.
\end{eqnarray*}
By~\eqref{basic-sum}, the sum over $k$ is $O(m^2)$. Moreover, the sum
over $j$ is a Riemann sum, and the function $x\mapsto 1/(1-\cos x)$ is
decreasing between $0$ and $\pi$, so that
$$
\sum_{j=M+1}^m \frac{1}{1-c_j} \le
\frac{m+1}\pi \int_{\frac{(2M+1)\pi}{2m+2}}^\pi \frac{dx}{1-\cos x} = 
\frac {m+1}{\pi\,\tan \frac{(2M+1)\pi}{4m+4}}= O\! \left( \frac{m^2}M\right).
$$
This gives
$$
\SUM_2= O\!\left( \frac{m^4}M\right)=o\,(m^4).
$$

Let us now focus on small values of $j$ and $k$.
The following estimates  hold uniformly in  $j$ and $ k$, when
$0\le j, k \le M$:
\begin{eqnarray}
  \frac{c_j+c_k}{(1-c_j)(1-c_k)}
&=& \frac{128 (m+1)^4}{(2j+1)^2(2k+1)^2\pi^4}
\left(1+ O\! \left( \frac{M^2 }{m^2}\right)\right),  \label{est1}
\\
&=& \frac{128 m^4}{(2j+1)^2(2k+1)^2\pi^4}
\left(1+ o(1)\right),\nonumber
\\
\nonumber\\
\left(1- \frac 4 m (1-c_jc_k)\right)^n 
&=& \exp\left( -\frac{n\pi^2((2j+1)^2+(2k+1)^2)}{2m(m+1)^2}\right)
\left(1+ O \!\left( n \, \frac{M^4}{m^5}\right)\right) \label{est2}
\\
&=& \exp\left( -\frac{n\pi^2((2j+1)^2+(2k+1)^2)}{2m^3}\right)
\left(1+ o(1)\right).\nonumber
\end{eqnarray}
We have used the fact that $n=O(m^3)$ and $M=o(\sqrt m)$.
%
Hence,
\begin{eqnarray*}
   \SUM_1 &:=&
\sum_{k,j=0}^M\frac{c_j+c_k}{(1-c_j)(1-c_k)} \,{x_{jk}}^n
\\
&=&
\frac{128 m^4}{\pi^4}
\left(\sum_{j=0}^M \frac {e^{ -\frac{n\pi^2(2j+1)^2}{2m^3}}}
 {(2j+1)^2}\right)  ^2
\left(1+ 
o(1)\right).
\end{eqnarray*}
Finally, observe that
$$
\sum_{j=0}^M \frac {e^{ -\frac{n\pi^2(2j+1)^2}{2m^3}}}
 {(2j+1)^2}= h\left(\frac n{m^3}\right) + O(e^{- \frac{n\, M^2}{m^3}
 }) = h\left(\frac n{m^3}\right) \left(1+o(1)\right)
$$
where
$$
h(\kappa)=\sum_{j=0}^M \frac {e^{ -{\kappa \pi^2(2j+1)^2}/{2}}}
 {(2j+1)^2}.
$$
Writing
$$
\II_{m,n}= \frac{m(m+1)}4 -\frac 1 {8(m+1)^2} \left( \SUM_1+\SUM_2\right)
$$
then yields the result.
\qed

\bigskip
\noindent{\em Proof of Proposition~\rm{\ref{prop:cubic-log}}.}
The first result follows from the second identity of Lemma~\ref{lem:bounds},
using
$$
n=cm^3\log m (1+o(1)) \quad \hbox{and} \quad  \log x_{00}=
-\frac{\pi^2}{m^3} (1+O(1/m)).
$$
 If $c<1/\pi^2$, we choose $\eps$ such that $\gamma:=1-c\pi^2-\eps
>0$, and the second result follows.
If $c>1/\pi^2$, we choose $\eps$ such that $\gamma:=c\pi^2-1-\eps
>0$, and the third result follows.

Assume now $n= 1/\pi^2 m^3 \log m + \alpha m^3 +o(m^3)$. We will show that the sum
over $j$ and $k$ occurring in~\eqref{expr-Imn-5} 
is dominated by the contribution obtained for $j=k=0$. This contribution reads
$$
\frac{2c_0}{(1-c_0)^2}{x_{00}}^n = \frac{128 m ^3}{\pi^4} e^{-\alpha
  \pi^2}(1+o(1)).
$$
The contribution  obtained for $j=k=m$ is
$$
-\frac{2c_0}{(1+c_0)^2}{x_{00}}^n = O(1/m).
$$
Finally, for all other values of $j$ and $k$, there holds $0<x_{jk}\le
x_{01}$. Hence
$$
\left|
\sum_{{0\le k,j\le m}\atop{(k,j)\not = (0,0),(m,m)} }
\frac{c_j+c_k}{(1-c_j)(1-c_k)} {x_{jk}}^n
\right|
\le
2\,{x_{01}}^n\sum_{0\le k,j\le m}\frac{1}{(1-c_j)(1-c_k)}= 2(m+1)^4
{x_{01}}^n
= O(1/m).
$$
We have first used the identity~\eqref{basic-sum}, then $n= 1/\pi^2
m^3 \log m+\alpha m^3+o(m^3)$,
and $\log x_{01}= -5\pi^2/m^3 (1+o(1))$.
The estimate of $\II_{m,n}$ follows.
\qed
\section{The intermediate regime}
\label{sec:int}
\begin{Proposition}
  Assume $m\ll  n_m \ll m^3$.
Then, denoting $n \equiv n_m$,
$$
\frac{\II_{m,n}}{\sqrt{mn}} \rightarrow \sqrt{\frac 2 \pi}.
$$
\end{Proposition}
\noindent{\bf Comment.} By Propositions~\ref{prop:linear}
and~\ref{prop:cubic}, one has
$$
\frac{\II_{m,n}}{\sqrt{mn}} \rightarrow \sqrt \kappa f(\kappa)
\quad \hbox{if } \quad n_m \sim \kappa m,
$$
while
$$
\frac{\II_{m,n}}{\sqrt{mn}} \rightarrow  \frac{ g(\kappa)}{\sqrt\kappa}
\quad \hbox{if } \quad n_m \sim \kappa m^3.
$$
It can be proved that
$$
\lim_{\kappa \rightarrow \infty} \sqrt \kappa f(\kappa)=\sqrt{\frac 2 \pi} 
=\lim_{\kappa \rightarrow 0}\frac{ g(\kappa)}{\sqrt\kappa},
$$
which makes all three regimes consistent.

Note that for a related continuous time chain, with inversion number
$D_m^t$,  it has been proved that when $m\ll t\ll m^3$, the random
variable $D_m^t/\sqrt{mt}$ converges in probability to $\sqrt{ 2/
  \pi}$~\cite{b-d-adjacent}.

\begin{proof}
The proof mixes arguments that we have already used for small and
large times. As in the 
small time case, we start from~\eqref{Imn-expr-3}. As in the large
time case, we split the sum over $j$ and $k$ into two parts: 
$j \le M$, $k\le M$ in the first part,
$j>M$ or $k>M$ in the second part. Here, $M=M_m$ is an integer
sequence that satisfies
$$
\sqrt{\frac{m^3}n} \ll M \ll \min \left(\frac{m^3}n, m \left( \frac m
n\right)^{1/4}\right).
$$
Such a sequence exists under the assumptions we have made on $n$.

We begin by bounding the second part. As in the large time case, one
easily obtains
\begin{eqnarray}  \left| \SUM_2 \right|&:=& \left|
\sum_{j \hbox{ \small or } k>M }
\frac{c_j+c_k}{(1-c_j)(1-c_k)} (1-{x_{jk}}^n)\right|
\nonumber \\
&\le&
 \sum_{j \hbox{ \small or } k>M }\frac{2}{(1-c_j)(1-c_k)}
\ =O\left(\frac{m^4}{M}\right).
\label{sum2-int}
\end{eqnarray}

We now focus on the case $j , k \le M$. The estimates~\eqref{est1}
and~\eqref{est2} still hold uniformly in $j$ and $k$.
 This gives
\begin{eqnarray*}
   \SUM_1 &:=&
\sum_{k,j=0}^M
\frac{c_j+c_k}{(1-c_j)(1-c_k)} (1-{x_{jk}}^n)
\\
&=&
\frac{128 m^4}{\pi^4}
\sum_{j,k=0}^M \frac {
1- \exp\left(-\frac{n\pi^2((2j+1)^2+(2k+1)^2)}{2m^3}\right)}
 {(2j+1)^2(2k+1)^2}
\left(1+ o(1)\right).
\end{eqnarray*}
Using
$$
1-e^{a+b}=(1-e^a)+(1-e^b) -(1-e^a)(1-e^b)
$$
this may be rewritten as
\beq\label{sum1-SM}
\SUM_1=\frac{128 m^4}{\pi^4}\left(
 2 S_M \sum_{k=0}^M\frac {1} {(2k+1)^2} -{S_M}^2
\right) \left(1+ o(1)\right)
\eeq
with 
$$
S_M=\sum_{j=0}^M \frac {1- \exp\left(
-\frac{n\pi^2(2j+1)^2}{2m^3}\right)} {(2j+1)^2}
.
$$
The sum $S_M$ is close to a Riemann sum. More precisely,
\begin{eqnarray}
S_M
&=&\frac{\pi\sqrt{n/m^3}}{2\sqrt 2 }
\int_0^{\sqrt 2 \pi M \sqrt{n/m^3} } \frac{1-e^{-x^2}}{x^2} dx
\left(1+ o(1)\right)
\nonumber \\
&=&\frac{\pi^{3/2}\sqrt{n/m^3}}{2\sqrt 2 }\left(1+ o(1)\right)
\nonumber 
\end{eqnarray}
 as 
$$
\int_0^{\sqrt 2 \pi M \sqrt{n/m^3} } \frac{1-e^{-x^2}}{x^2}dx=
\int_0^{\infty} \frac{1-e^{-x^2}}{x^2} dx \left(1+ o(1)\right)
=
\sqrt \pi + o(1).
$$
In particular, $S_M=o(1)$. Returning to~\eqref{sum1-SM}, let us finally note that 
$$
\sum_{k=0}^M\frac {1} {(2k+1)^2}= \frac{\pi^2}8 \left(1+ o(1)\right).
$$
This gives
$$
\SUM_1= {\frac{16}{\sqrt{2\pi}}}\, m ^2 \sqrt {mn} \left(1+ o(1)\right).
$$
 Now recall  that
$$
\II_{m,n}= \frac 1 {8(m+1)^2} (\SUM_1+\SUM_2).
$$
 Combining the above estimate of $\SUM_1$ and the
 bound~\eqref{sum2-int} on $\SUM_2$ gives the announced estimate of 
$
\II_{m,n}.
$
\end{proof}

\section{Perspectives}
Many  interesting Markov chains on groups have been studied,
and it is natural to ask to which similar problems the approach used
in this
note could be adapted. To make this question more precise,  let us
underline that such problems may involve changing the dynamics of the
chain, changing the statistics under consideration,
or changing the underlying group.

\medskip\noindent 
{\bf Changing the chain.} Already in the symmetric group $\Sn_{m+1}$, a
number of different dynamics have been considered. One can for instance
multiply at each step by \emm any,
transposition~\cite{b-d-general,diaconis-shahshahani-transpositions,schramm-composition}, or consider
\emm cyclic, \ps\ by allowing multiplications by the transpositions $(0,1),
(1,2), \ldots, (m-1,m)$ and $ (m,0)$, or only allow transpositions with the
largest element~\cite{flatto,diaconis-saloff-comparison}, or perform \emm block,
transpositions~\cite{bona-flynn}, or \emm shuffle, the elements in various
ways~\cite{bayer-diaconis}, etc. See for
instance~\cite{diaconis-cutoff}.

\medskip\noindent 
{\bf Changing the statistics.} We have focused in this paper on the
inversion number. It is an interesting parameter from a probabilistic
point of view, as it gives an indication of whether the chain may be
mixed at time $n$ or not. However, in certain biological contexts, it may be more
sensible to estimate instead the natural ``distance'' between
$\pi^{(0)}$ and $\pi^{(n)}$, defined as the minimum number of chain steps that
lead from one to the other~\cite{b-d-general,b-d-adjacent}. For the
chain studied here, this 
coincides with the inversion number, but for other dynamics it will be
a different parameter. For instance, if we multiply by any
transposition, one step suffices to go from $012$ to $210$, whereas
the inversion number of $210$ is 3. Clearly, the approach we have used here
relies heavily on the possibility of describing in simple terms the
evolution of the parameter under consideration (as we did by the
combination of~\eqref{inv-pijn} and Lemma~\ref{lem:triangle}). 
Note that an explicit expression for the expected distance after $n$
(non-necessarily adjacent) transpositions has been given in~\cite{eriksen-hultman}.

\medskip\noindent 
{\bf Changing the group.} Although many different chains in many
different groups have been considered, we are
primarily thinking of classical families
of (finite or affine) Coxeter groups, because  the inversion number
admits in these groups  a
natural generalization (the \emm
length,) which  has usually a simple description~\cite[Chap.~8]{bjorner-brenti-book}. According
to~\cite{niklas}, an explicit formula, due to Troili, is 
already known for the average length of the product of $n$ generators in the
affine group $\tilde A_m$. Another solved case is the hypercube $\zs_2^m$~\cite{diaconis-hypercube}.

\bigskip
\noindent{\bf Acknowledgements.} The author thanks an anonymous
referee, Niklas Eriksen and David Wilson for their suggestions, comments, and
help with the bibliography.

\bibliographystyle{plain}
\bibliography{biblio.bib}

\begin{thebibliography}{10}

\bibitem{aldous}
D.~Aldous.
\newblock Random walks on finite groups and rapidly mixing {M}arkov chains.
\newblock In {\em Seminar on probability, {XVII}}, volume 986 of {\em Lecture
  Notes in Math.}, pages 243--297. Springer, Berlin, 1983.

\bibitem{bayer-diaconis}
D.~Bayer and P.~Diaconis.
\newblock {Trailing the dovetail shuffle to its lair.}
\newblock {\em Ann. Appl. Probab.}, 2(2):294--313, 1992.

\bibitem{b-d-general}
N.~Berestycki and R.~Durrett.
\newblock {A phase transition in the random transposition random walk.}
\newblock {\em Probab. Theory Relat. Fields}, 136(2):203--233, 2006.

\bibitem{b-d-adjacent}
N.~Berestycki and R.~Durrett.
\newblock {Limiting behavior of the distance of a random walk}.
\newblock {\em Electronic J. Prob.}, 13(2):374--395, 2008.

\bibitem{bjorner-brenti-book}
A.~Bj\"orner and F.~Brenti.
\newblock {\em Combinatorics of Coxeter groups}, volume 231 of {\em Graduate
  Texts in Mathematics}.
\newblock Springer, 2005.

\bibitem{bona-flynn}
M.~B\'ona and R.~Flynn.
\newblock The average number of block interchanges needed to sort a permutation
  and a recent result of {S}tanley.
\newblock {\em Inform. Process. Lett.}, 109(16):927--931, 2009.

\bibitem{bousquet-motifs}
M.~Bousquet-M{\'e}lou.
\newblock Four classes of pattern-avoiding permutations under one roof:
  generating trees with two labels.
\newblock {\em Electron. J. Combin.}, 9(2):Research paper 19, 31 pp.
  (electronic), 2002/03.

\bibitem{Bous05}
M.~Bousquet-M{\'e}lou.
\newblock Walks in the quarter plane: {K}reweras' algebraic model.
\newblock {\em Ann. Appl. Probab.}, 15(2):1451--1491, 2005.

\bibitem{bousquet-mishna}
M.~Bousquet-M\'elou and M.~Mishna.
\newblock Walks with small steps in the quarter plane.
\newblock Arxiv:0810.4387, 2008.

\bibitem{bousquet-xin}
M.~Bousquet-M{\'e}lou and G.~Xin.
\newblock On partitions avoiding 3-crossings.
\newblock {\em S\'em. Lothar. Combin.}, 54:Art. B54e, 21 pp. (electronic),
  2005/07.

\bibitem{diaconis-book}
P.~Diaconis.
\newblock {\em {Group representations in probability and statistics.}},
  volume~11 of {\em IMS Lecture Notes-Monograph Series}.
\newblock {Institute of Mathematical Statistics}, 1988.

\bibitem{diaconis-cutoff}
P.~Diaconis.
\newblock {The cutoff phenomenon in finite Markov chains.}
\newblock {\em Proc. Natl. Acad. Sci. USA}, 93(4):1659--1664, 1996.

\bibitem{diaconis-hypercube}
P.~Diaconis, R.~L. Graham, and J.~A. Morrison.
\newblock Asymptotic analysis of a random walk on a hypercube with many
  dimensions.
\newblock {\em Random Struct. Algorithms}, 1:51--72, 1990.

\bibitem{diaconis-saloff-comparison}
P.~Diaconis and L.~Saloff-Coste.
\newblock {Comparison theorems for reversible Markov chains.}
\newblock {\em Ann. Appl. Probab.}, 3(3):696--730, 1993.

\bibitem{diaconis-shahshahani-transpositions}
P.~Diaconis and M.~Shahshahani.
\newblock {Generating a random permutation with random transpositions.}
\newblock {\em Z. Wahrscheinlichkeitstheor. Verw. Geb.}, 57:159--179, 1981.

\bibitem{durrett-book}
R.~Durrett.
\newblock {\em Probability models for {DNA} sequence evolution}.
\newblock Probability and its Applications (New York). Springer, New York,
  second edition, 2008.

\bibitem{edelman-white}
P.~H. Edelman and D.~White.
\newblock Codes, transforms and the spectrum of the symmetric group.
\newblock {\em Pacific J. Math.}, 143(1):47--67, 1990.

\bibitem{niklas}
N.~Eriksen.
\newblock {Expected number of inversions after a sequence of random adjacent
  transpositions -- an exact expression.}
\newblock {\em Discrete Math.}, 298(1-3):155--168, 2005.

\bibitem{eriksen-hultman}
N.~Eriksen and A.~Hultman.
\newblock Estimating the expected reversal distance after a fixed number of
  reversals.
\newblock {\em Adv. Applied Math.}, 32:439--453, 2004.

\bibitem{eriksson}
H.~Eriksson, K.~Eriksson, and J.~Sj{\"o}strand.
\newblock Expected number of inversions after a sequence of random adjacent
  transpositions.
\newblock In {\em Formal power series and algebraic combinatorics ({M}oscow,
  2000)}, pages 677--685. Springer, Berlin, 2000.
\newblock ArXiv:math.co/0411197.

\bibitem{flajolet-stacks}
Ph. Flajolet.
\newblock The evolution of two stacks in bounded space and random walks in a
  triangle.
\newblock In J.~Gruska, B.~Rovan, and J.~Wiedermann, editors, {\em 12th
  Symposium on Mathematical Foundations of Computer Science}, volume 233 of
  {\em Lecture Notes in Computer Science}, pages 325--340, 1986.

\bibitem{flatto}
L.~Flatto, A.~M. Odlyzko, and D.~B. Wales.
\newblock {Random shuffles and group representations.}
\newblock {\em Ann. Probab.}, 13:154--178, 1985.

\bibitem{Rensburg-Prellberg-Rechni}
E.~J. Janse~van Rensburg, T.~Prellberg, and A.~Rechnitzer.
\newblock Partially directed paths in a wedge.
\newblock {\em J. Combin. Theory Ser. A}, 115(4):623--650, 2008.

\bibitem{knuth}
D.~E. Knuth.
\newblock {\em The art of computer programming, Vol. 1: Fundamental
  algorithms}.
\newblock Addison-Wesley Series in Computer Science and Information Processing,
  Reading, Mass.-London-Amsterdam, second edition, 1973.

\bibitem{Mishna-Rechni}
M.~Mishna and A.~Rechnitzer.
\newblock Two non-holonomic lattice walks in the quarter plane.
\newblock {\em Theoret. Comput. Sci.}, 410(38-40):3616--3630, 2009.
\newblock Arxiv:math/0701800.

\bibitem{saloff-coste-notices}
L.~Saloff-Coste.
\newblock {Probability on groups: Random walks and invariant diffusions.}
\newblock {\em Notices Am. Math. Soc.}, 48(9):968--977, 2001.

\bibitem{saloff-coste-survey}
L.~Saloff-Coste.
\newblock Random walks on finite groups.
\newblock In {\em Probability on discrete structures}, volume 110 of {\em
  Encyclopaedia Math. Sci.}, pages 263--346. Springer, Berlin, 2004.

\bibitem{schramm-composition}
O.~Schramm.
\newblock Compositions of random transpositions.
\newblock {\em Israel J. Math.}, 147:221--243, 2005.
\newblock Arxiv:math/0404356v3.

\bibitem{wang-warnow}
L.~S. Wang and T.~Warnow.
\newblock Estimating true evolutionary distances between genomes.
\newblock In {\em Proceedings of the 33rd Annual ACM Symposium on the Theory of
  Computing (STOC'01)}, pages 637--646, Crete, Greece, 2001.

\bibitem{wilson}
D.~B. Wilson.
\newblock Mixing times of lozenge tiling and card shuffling {M}arkov chains.
\newblock {\em Ann. Appl. Probab.}, 14(1):274--325, 2004.
\newblock ArXiv math.PR/0102193.

\bibitem{xin-zhang}
G.~Xin and T.~Y.~J. Zhang.
\newblock Enumeration of bilaterally symmetric 3-noncrossing partitions.
\newblock {\em Discrete Math.}, 309(8):2497--2509, 2009.
\newblock Arxiv:0810.1344.

\end{thebibliography}

\end{document}